\documentclass[11pt,letterpaper]{amsart}

\usepackage{graphicx}
\usepackage{amssymb}
\usepackage[french, english]{babel}
\usepackage[utf8]{inputenc}

\usepackage{amsmath,amsfonts,amssymb,amsthm,amsbsy,mathrsfs}
\usepackage[T1]{fontenc}

\usepackage{mathrsfs}
\usepackage{dsfont}
\usepackage{tikz}
\usepackage{bbm}
\usepackage{appendix}
\usepackage{pifont}
\usepackage[normalem]{ulem}
\usetikzlibrary{patterns}
\usepackage[colorlinks]{hyperref}
\usepackage{cleveref}
\definecolor{darkgreen}{rgb}{0.0, 0.5, 0.0}

\newcommand{\eps}{\varepsilon}

\newcommand{\vol}{\text{vol}}

\newcommand{\bT}{\mathbf{T}_{2n,g_n}}

\renewcommand{\P}{\mathbb P}

\newcommand{\somme}[3]{\sum_{#1}^{#2} #3}

\newcommand{\indi}[1]{\mathds{1}_{#1}}

\newcommand{\tk}{\frac{1}{4D}}

\theoremstyle{definition}

\newtheorem{thm}{Theorem}

\newtheorem{defn}{Definition}
\newtheorem{rem}[defn]{Remark}

\newtheorem{prop}[defn]{Proposition}

\newtheorem{lem}[defn]{Lemma}

\usepackage{todonotes}
\usepackage{comment}

\crefname{thm}{Theorem}{Theorems}
\crefname{lem}{Lemma}{Lemmas}
\crefname{prop}{Proposition}{Propositions}
\crefname{defn}{Definition}{Definitions}

\def\isol{$\kappa$-isolated}
\def\strongisol{strongly $\kappa$-isolated}
\def\keps{\ensuremath{\kappa_\eps}}
\def\kepsvalue{\ensuremath{(1-\eps)\kappa}}
\def\isolbis{$\keps$-isolated}
\def\badset{$\kappa$-bad set}
\def\strongbadset{strong $\kappa$-bad set}
\def\badsetbis{$\keps$-bad set}

\def\kexp{$\kappa$-expander}
\def\kexpbis{$\keps$-expander}
\def\J{\mathcal J}

\title[High genus expanders]{\bf{Large expander subgraphs in high genus triangulations}}

\author{Tanguy {Lions}}
\thanks{ENS Lyon, \url{tanguy.lions@ens-lyon.fr}}

\author{Baptiste {Louf}}
\thanks{CNRS and Université de Bordeaux, \url{baptiste.louf@math.u-bordeaux.fr}}

\begin{document}

\maketitle

\begin{abstract}
We prove that random triangulations of high genus contain very large expander subgraphs, answering a question of Benjamini. Our approach relies on new general criteria for arbitrary graphs to contain large expander subgraphs.
\end{abstract}

\section{Introduction and main results}
In this paper, we dive further into the large scale geometry of \emph{high genus maps}. 

\emph{Combinatorial maps} are a model of discrete surfaces constructed by gluing polygons along their edges. They have been well studied for several decades now, especially in the planar case (maps of the sphere): \emph{enumeration} results go back to Tutte in the 60's \cite{Tut62,Tut62b,Tut63}, then a \emph{bijective approach} has been carried out~\cite{CV81,Sch98these,BDG04}. With all these tools in hand, the study of \emph{geometric properties of large random planar maps} soon followed~\cite{AS03,CS04,LG11, Mie11}. Similar results hold for fixed positive genus~\cite{Chapuy:profile, BM22}. 

It is then natural to ask the same kind of questions as \emph{both the genus and the size go to infinity}, which turns out to be challenging due to the difficulty to obtain enumerative results, as we are now in the realm of \emph{multivariate asymptotics}, which is a challenging and wide open topic (see e.g.~\cite{acsv_web}). Thanks to a powerful bijection~\cite{CFF13}, a specific family, \emph{unicellular maps}, was investigated quite systematically in high genus \cite{ACCR13,Ray13a, Louf-expander, JansonLouf1, JansonLouf2}. 

Finally, recently, several results on more general models of high genus maps (mostly triangulations) were obtained: first their local behaviour~\cite{BL19,BL20}, then some global observables~\cite{Louf,BudLoufChapuy,Louf-systole}. All of these works shed light on the \emph{hyperbolic} nature of high genus maps, notably by analogy with another model of high genus surfaces: \emph{Weil--Petersson hyperbolic surfaces} (see e.g.~\cite{Mir13,JansonLouf2}).

In particular, and this will be the main focus of this work, hyperbolic behaviour implies good \emph{expansion properties} (for instance, hyperbolic surfaces are expanders~\cite{Mir13}). However, maps cannot be perfect expanders due to their fractal nature. Indeed, by~\cite{BL19}, every constant size pattern (including those that have very bad expansion) will appear somewhere in the map, hence the Cheeger constant has to be $o(1)$, and it is actually even of order $\frac{1}{\log n}$ as shown in~\cite{BudLoufChapuy}. However, Benjamini conjectured that high genus maps must contain a \emph{large expander subgraph}. This was first shown for unicellular maps by the second author in~\cite{Louf-expander}, and in this paper we show that it is also true for high genus triangulations.

\begin{thm}\label{thm_triangulations}
    Fix a sequence $(g_n)_{n\geq0}$ such that $\frac{g_n}{n}\to \theta \in(0,1/2)$. Let $\bT$ be a uniform triangulation of genus $g_n$ on $3n$ edges, and $\bT^*$ its dual map. Then for every $\eps>0$, there exists $\kappa>0$ depending only on $\theta$ and $\eps$ such that with probability $\geq 1-\eps$, both $\bT$ and $\bT^*$ contain a submap on at least $(1-\eps)3n$ edges that is a \kexp{}.
\end{thm}

\begin{rem}
    We conjecture that this theorem still holds if we replace "with probability $\geq 1-\eps$" by "with probability $1-o(1)$".
\end{rem}

In order to prove this theorem, we will start from a result of~\cite{BudLoufChapuy} that bounds the number of "\emph{isolated faces}" in high genus triangulations. Roughly speaking, in any graph, the \isol{} vertices are those that prevent it from being a \kexp{} (and isolated faces are defined by duality). We refer to~\cref{sec_def} for precise definitions.

The rest of the proof will rely on two new general theorems that establish criteria to find large expander subgraphs in arbitrary graphs. This question has been tackled in several recent articles~\cite{bottcher2010bandwidth,  Kri18, krivelevich2019expanders, LS20,chakraborti2022well,LMS}.

First, we show that having few isolated vertices  means having a large expander subgraph. We phrase this theorem in terms of volume, i.e. the sum of the degrees.
\begin{thm}\label{thm_isolated_to_expander}
    Let $G$ be a graph on $n$ edges such that the total volume of its \isol{} vertices is less than $\eps n$, with $0<\eps<1/2$. Then $G$ contains an induced subgraph on at least $(1-\eps)n$ edges that is a \kexpbis{}, where $\keps=\kepsvalue$.
\end{thm}
\begin{rem}
    We leave the question of a reciprocal result open. That is, does having an expander subgraph (plus possibly some mild extra assumptions) imply having few isolated vertices ?
\end{rem}

Finally, we show that, provided that the face degrees are bounded, the property of having a large expander subgraph is stable under duality of maps.

\begin{thm}\label{thm_dual}
  Fix $\eps,\kappa,D>0$.  Let $M$ be a map on $n$ edges with face degrees bounded by $D$ such that its dual map $M^*$ contains an induced subgraph on at least $(1-\eps)n$ edges that is a \kexp{}. Then $M$ contains a subgraph on at least $(1-\eps)n$ edges that is a $\frac{\kappa}{8D}$-expander.
\end{thm}

\subsection*{Acknowledgements.} 

We thank Thomas Budzinski and Guillaume Chapuy for (joint) works and inspiring discussions about the geometry of high genus maps. TL thanks Justin Salez for interesting discussions on expanders. BL also thanks Fiona Skerman for their joint works and discussions around large expander subgraphs. Finally, we thank Itai Benjamini for asking the question that led to~\cref{thm_triangulations}.

This work was started during the workshop "Combinatorics and Discrete Probability" organized at CIRM in November 2025.

TL was supported by ANR IsOMa (ANR-21-CE48-0007). BL was partially supported by ANR CartesEtPlus (ANR-23-CE48-0018) and ANR HighGG (ANR-24-CE40-2078-01).

\section{Definitions}\label{sec_def}

\subsection{Graphs}
In the rest of the paper, we denote by $G$ a multigraph, that is a graph with loops and multiple edges allowed. For $X,Y \subset V(G)$, we denote by $E_{G}(X,Y)$ for the set of oriented edges in $G$ with a starting point in $X$ and an endpoint in $Y$. We also write $$e_G(X,Y) = \# E_G(X,Y).$$ Then for $x \in G$, we denote by $\deg_{G}(x) = e(\{x\},G \setminus \{x\})$. 
Finally, for $X  \subset V(G)$ we define the volume of the graph $X$ seen as a subgraph of $G$ as follows:
\begin{align}\label{volume}
    \vol_G(X) = \somme{x \in X}{}{\deg_G(x)}.
\end{align}
Sometimes, we forget the subscript $G$ when the graph $G$ is obvious. Then, we simply write
\begin{center}
    $\vol(X) := \vol_G(X)$.
\end{center}
The \textbf{Cheeger constant} or \textbf{isoperimetric number} is given by :
\begin{align}\label{cheeger_constant_definition}
h(G) = \inf \bigg\{ \frac{e_G(X,G\setminus X )}{\vol_G(X)}  \text{: } X \subset G \text{ and } \vol_G(X) \le \frac{\vol_G(G)}{2}\bigg \}.
\end{align}

\begin{defn}\label{definition_badsets}
    Take $\kappa > 0$ and $X \subset G$ connected\footnote{that is, the subgraph of $G$ induced by $X$ is connected.}. We say that $X$ is a \emph{\badset{}} if $\vol_G(X) \le \frac{\vol_G(G)}{2}$ and $\frac{e_G(X,G\setminus X )}{\vol_G(X)}  < \kappa$ (see Figure~\ref{badset}). We say that a \badset{} $X$ is a \emph{\strongbadset{}} if  $G \setminus X$ is also connected.
\end{defn}

 \begin{figure}[h]
     \centering
     \includegraphics[scale=0.2]{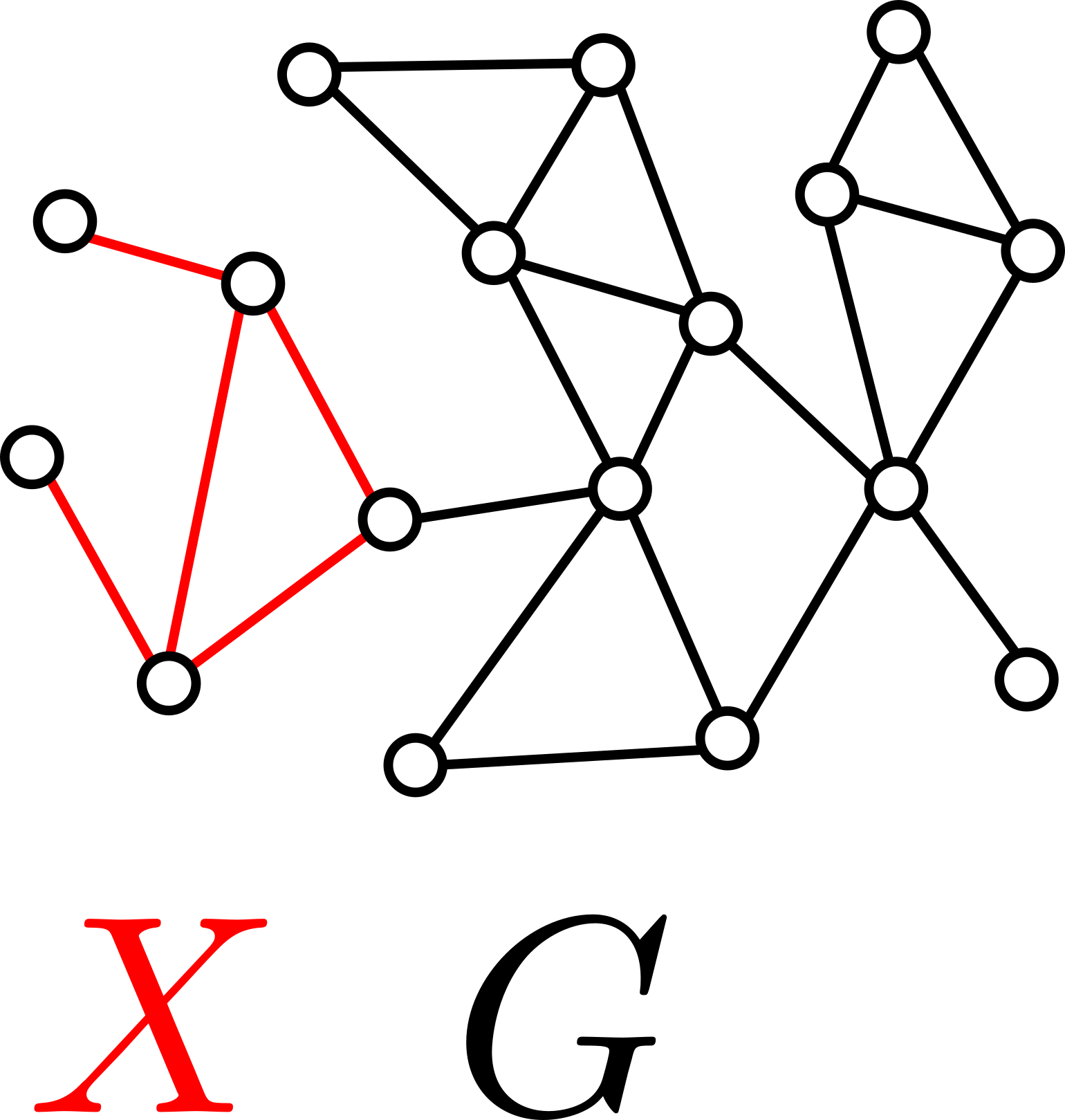}
     \caption{On this example, we represent a graph $G$ and $X \subset G$ connected that is a $\kappa$-bad set with $\kappa = \frac{1}{10}$. It is also a strong $\kappa$-bad set. }
     \label{badset}
 \end{figure}

\begin{defn}\label{definition_expander}
For any $\kappa > 0$, we say that $G$ is a \emph{$\kappa$-expander} if $h(G) \ge \kappa$ or equivalently if $G$ does not contain a \badset{}.
\end{defn}

\begin{defn}\label{isolated_vertex}
Take $\kappa > 0$ and $v \in G$. We say that $v$ is \emph{\isol{}} (resp. \emph{\strongisol{}}) if there exists a \badset{} (resp. \strongbadset) $X$  such that $v \in X$. We denote by $\mathbf{Isol}_{\kappa}(G)$ (resp. $\mathbf{Isol}^{+}_{\kappa}(G)$) the set of $\kappa$-isolated (resp. \strongisol{}) vertices. 
\end{defn}

Note that these two notions are very similar.
In the next lemma, we show that having few \strongisol{} vertices implies having few $\kappa^{2}$-isolated vertices.

\begin{lem}\label{strong_isolated_to_isolated}
Take  $0<\eps<1$ and $\kappa>0$ such that $\kappa <\min\left(\frac{1-2\varepsilon}{3},1-4\eps\right)$. If the connected (multi)graph $G$ satisfies $ \vol_{ {G}}(\mathbf{Isol}^{+}_{\kappa}(G)) \le \varepsilon \vol(G)$, then we have $\mathbf{Isol}_{\kappa^2}(G) \subset \mathbf{Isol}_{\kappa}^{+}(G)$, and in particular $ \vol_{ {G}}(\mathbf{Isol}_{\kappa^2}(G)) \le \varepsilon \vol(G)$.
\end{lem}
\begin{proof}
Let $v \in G$ be a $\kappa^2$-isolated vertex and let  $X$ be an arbitrary $\kappa^2$-bad set it belongs to. We write $C_1,\dots,C_r$ for the connected components of $G \setminus X$. We assume that $\vol_{ {G}}(C_1) \ge \cdots \ge \vol_{ {G}}(C_r)$. Thus, we have
\begin{align*}
e_{ {G}}(X,G \setminus X) = \somme{i=1}{r}{e_{ {G}}(C_i,G \setminus C_i) }.
\end{align*}

Let us write $Y = X \cup C_2 \cup \dots \cup C_{r}$. Thus, we have $v \in Y$, and to establish the lemma, we want to show that $Y$ is a \strongbadset. It is easily verified that  $\frac{e_{ {G}}(Y,G \setminus Y) }{\vol_{ {G}}(Y)} \le \frac{e_{ {G}}(X,G \setminus X) }{\vol_{ {G}}(X)} \le \kappa^2<\kappa$ and $G \setminus Y = C_1$ is connected (see Figure~\ref{prooflemma6}).

\begin{figure}[h]
    \centering
    \includegraphics[scale=0.2]{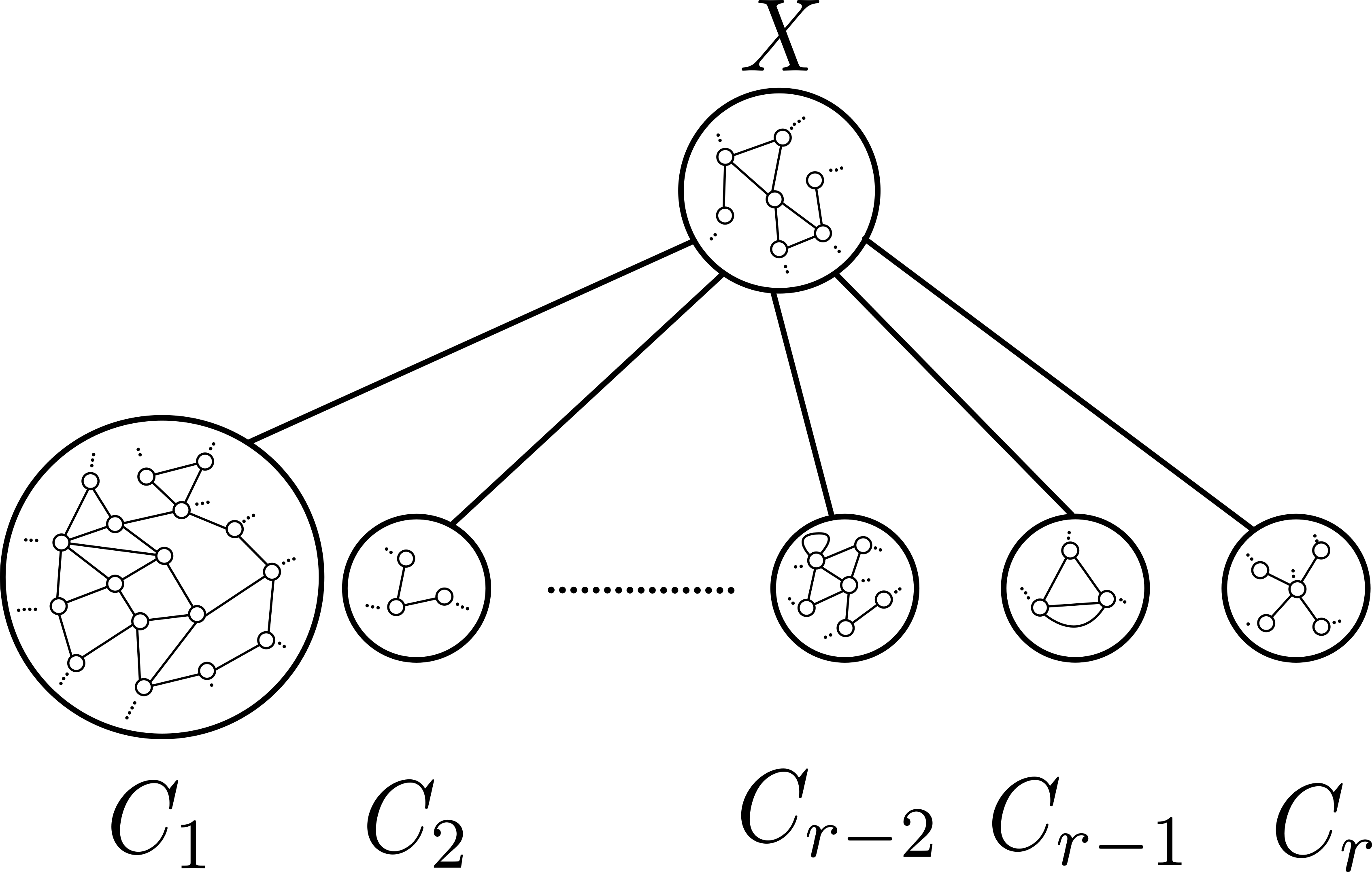}
    \caption{On this example, we represent on top the subset $X \subset G$ and on the bottom we represent the connected components $C_1,\cdots,C_r$ of $G \setminus X$.}
    \label{prooflemma6}
\end{figure}

Let us now verify that $\vol_{ {G}}(Y) \le \frac{\vol(G)}{2}$. We have 
\begin{align*}
\vol_{ {G}}(Y) = \vol_{ {G}}(X) + \somme{i=2}{r}{\vol_{ {G}}(C_i)}.
\end{align*}
Then, for any $2 \le i \le r$, using the fact that we have $\vol_{ {G}}(C_i) \le \frac{\vol(G)}{2}$ (since $C_1$ is the largest component) either $\frac{e_{ {G}}(C_i,G \setminus C_i) }{\vol_{ {G}}(C_i)} \ge \kappa$ or $C_i \subset \mathbf{Isol}^{+}_{\kappa}(G)$. Thus, we deduce that the right-hand side of the previous equality is bounded by 
\begin{align*}
&\vol_{ {G}}(X) + \somme{i=2}{r}{\vol_{ {G}}(C_i)\indi{\frac{e_{ {G}}(C_i,G \setminus C_i) }{\vol_{ {G}}(C_i)} \ge \kappa}} + \somme{i=2}{r}{\vol_{ {G}}(C_i) \indi{C_i \subset \mathbf{Isol}^{+}_{\kappa}(G) } } \\&\le \vol_{ {G}}(X)+\kappa^{-1}e_{ {G}}(X,G \setminus X) + \vol_{ {G}}(\mathbf{Isol}^{+}_{\kappa}(G)) \\
&\le (1+\kappa )\vol_{ {G}}(X)+\varepsilon \vol(G)\\
&\le \bigg(\frac{1+\kappa}{2}+\varepsilon\bigg)\vol(G).
\end{align*}
This rewrites 
\begin{align*}
\vol_{ {G}}(Y) \le \bigg(\frac{1+\kappa}{2}+\varepsilon\bigg)\vol(G).
\end{align*}
Let us reason by contradiction and assume that $\vol_{ {G}}(Y) > \frac{\vol(G)}{2}$. Then we have 
\begin{align*}
\frac{\vol(G)}{2} \ge \vol_{ {G}}(C_1) = \vol(G)-\vol_{ {G}}(Y) \ge \bigg(\frac{1}{2}-\frac{\kappa}{2}-\varepsilon\bigg)\vol(G).
\end{align*}
Using this and the fact that $\kappa < \frac{1-2\varepsilon}{3}$, we obtain 
\begin{align*}
\displaystyle \frac{e_{ {G}}(C_1,G \setminus C_1)}{\vol_{ {G}}(C_1)} \le \bigg(\frac{1}{2}-\frac{\kappa}{2}-\varepsilon\bigg)^{-1}\frac{e_{ {G}}(X,G \setminus X)}{\vol_{ {G}}(X)} \le \frac{\kappa^2}{\frac{1}{2}-\frac{\kappa}{2}-\varepsilon} <\kappa.
\end{align*} 
We deduce that $C_1 \subset \mathbf{Isol}^{+}_{\kappa}(G)$. Thus
\[\varepsilon \vol(G)\ge \vol_{ {G}}( \mathbf{Isol}^{+}_{\kappa}(G)) \ge \vol_{ {G}}(C_1) \ge \left(\frac{1}{2}-\frac{\kappa}{2}-\varepsilon\right)\vol(G).\] 
But it is easily checked that this contradicts the assumption that $\kappa<1-4\eps$. Therefore, by contradiction,  $\vol_{ {G}}(Y) \le \frac{\vol(G)}{2}$ and $Y$ is a \strongbadset. Since $v \in Y$, $v$ is strongly $\kappa$-isolated. This concludes the proof.
\end{proof}

\subsection{Maps and triangulations}

A \emph{map} $m$ is a finite graph (possibly containing loops and multiple edges) embedded in a compact connected oriented surface, considered up to homeomorphism. The connected components of the complement of the graph on the surface are called the \emph{faces} of $m$. Equivalently, a map can be viewed as a connected oriented surface obtained by identifying the sides of a finite collection of polygons, where each polygon corresponds to a face.

The \emph{genus} of a map is the genus of its underlying surface. The maps considered here are always \emph{rooted}, meaning they are equipped with a distinguished oriented edge called the \emph{root edge}. The vertex at the origin of the root edge is the \emph{root vertex}, and the face to its right is the \emph{root face}.

The \emph{degree} of a face is the number of incident edge-sides. Note that if both sides of an edge are incident to the same face, the edge contributes twice to the degree. A \emph{triangulation} is a rooted map in which all faces have degree $3$. For $n \geq 1$ and $g \geq 0$, let $\mathcal{T}(2n,g)$ denote the set of triangulations of genus $g$ with $2n$ faces (an even number of faces is required for edges to be glued in pairs).

According to Euler's formula, a triangulation in $\mathcal{T}(2n,g)$ has $3n$ edges and $n+2-2g$ vertices. In particular, the set $\mathcal{T}(2n,g)$ is non-empty if and only if $n \geq 2g-1$. We denote by $\mathbf{T}_{2n,g}$ a uniform random variable on $\mathcal{T}(2n,g)$.
\section{Proof of the main result: \cref{thm_triangulations}}

In this section, we prove our main theorem, using our structural results~\cref{thm_isolated_to_expander,thm_dual} whose proofs are postponed to the next sections.

Fix $\varepsilon > 0$. In \cite[Proposition $12$]{BudLoufChapuy} the authors prove that, most likely, $\bT$ contains few \strongisol{} vertices, that is there exists $\kappa_0 > 0$ such that
\begin{align}\label{lower_bound_prob_isol}
    \P\left(\#\mathbf{Isol}^{+}_{\kappa_0}(\mathbf{T}_{2n,g_n}^{*}) \ge \eps n\right) \le\eps.
\end{align}
Now, let us work under the event $\#\mathbf{Isol}^{+}_{\kappa_0}(\mathbf{T}_{2n,g_n}^{*}) \le \eps n$. The rest of the reasoning is entirely deterministic. We recall that $\mathbf{T}_{2n,g_n}^{*}$ is a $3$-regular graph with $2n$ vertices and $3n$ edges. Thus its total volume is 
$\vol(\mathbf{T}_{2n,g_n}^{*}) = 6n$ and

\[\vol\left(\mathbf{Isol}^{+}_{\kappa_0}(\mathbf{T}_{2n,g_n}^{*})\right) = 3\#\mathbf{Isol}^{+}_{\kappa_0}(\mathbf{T}_{2n,g_n}^{*}).\] We deduce that we have 
\[\vol\left( \mathbf{Isol}^{+}_{\kappa_0}(\mathbf{T}_{2n,g_n}^{*})\right) \le 3\varepsilon n .\] Up to reducing $\kappa_0$, we may assume that $\kappa_0 \le \min\left(\frac{1-2\varepsilon}{3},1-4\eps\right)$. Using Lemma~\ref{strong_isolated_to_isolated}, we deduce 
\begin{align}\label{bound_isolated_vertex}
    \vol\left(\mathbf{Isol}_{\kappa_0^2}(\mathbf{T}_{2n,g_n}^{*})\right) \le  3\varepsilon n.
\end{align} 

Then, using Theorem~\ref{thm_isolated_to_expander}, we deduce that there exists $G^{*} \subset \mathbf{T}_{2n,g_n}^{*}$ with at least $(1-\varepsilon)3n$ edges such that $G^{*}$ is a $\keps$-expander with $\kappa_{\eps} = (1-\varepsilon)\kappa_0^2$. Moreover, $\mathbf{T}_{2n,g_n}$ is a map with $3n$ edges and faces of degree $3$. Using Theorem~\ref{thm_dual} with $D=3$, we deduce that $\mathbf{T}_{2n,g_n}$ contains a subgraph $G$ on at least $(1-\varepsilon)3n$ edges that is a $\frac{1}{24}\kappa_{\varepsilon}$-expander. This concludes the proof by taking $\kappa = \frac{1}{24}\kappa_{\varepsilon}$.
\section{From isolated vertices to expander subgraphs: proof of \cref{thm_isolated_to_expander}}
Recall that we start with a graph $G$ on $n$ edges whose \isol{} vertices have total volume at most $\eps n$  and recall that $\keps=\kepsvalue<\kappa$. Let us define a finite process to "remove bad vertices from $G$".
We start with $G_0=G$. At each step $i\geq 1$, either $G_{i-1}$ is a \kexpbis, or it is empty, and we stop the process, or there exists $S_{i}\subset V(G_{i-1})$ that is a \badsetbis{} of $G_{i-1}$. Then we set $G_i=G_{i-1}\setminus S_i$. 

This is a finite process, let us call $\tau$ its final step. The final graph $G_\tau$ is an induced subgraph of $G$ that is either empty or a \kexp{}. To establish~\cref{thm_isolated_to_expander}, we wish to show that it actually has at least $(1-\eps)n$ edges.

By definition of the process, $G_\tau$ has at least $n-\sum_i \vol_{ {G}}(S_i)$ edges.
Therefore, since we made the assumption that $\vol_{ {G}}(\mathbf{Isol}_{\kappa}(G))\leq \eps n$,~\cref{thm_isolated_to_expander} will be an immediate consequence of the following:
\begin{prop}\label{prop_process}
   The vertices of the $S_i$ are all \isolbis{}, that is $\bigcup_{1\leq i\leq\tau}S_i\subset \mathbf{Isol}_{\keps}(G)\subset \mathbf{Isol}_{\kappa}(G)$.
\end{prop}

\begin{proof}
    
Let us reason by induction.
For $t=1$, the proof is pretty clear, since $S_1$ itself is a \badsetbis{} of $G$.

Now let us turn to  $t>1$.  For $1\leq i,j\leq t $, set $A_{i,j}=e_{ {G}}(S_i,S_j)$ and $B_i=e_{ {G}}(S_i,G_t)$.  Let $\mathcal J$ be the "connected component of $S_t$", i.e. the set of indices $i\leq t$ such that $S_i$ is connected to $S_t$ in $G\setminus G_t$ (see Figure~\ref{proof_section4}). To shorten notation, we also write
$S_\J=\bigcup_{i\in \J}S_i$, as well as $B_\J=e_{ {G}}(S_\J,G_t)$. Recall that all $S_i$'s are disjoint. We want to show that $S_\J$ is a \badsetbis{}.

\begin{figure}[h]
    \centering
    \includegraphics[scale =0.2]{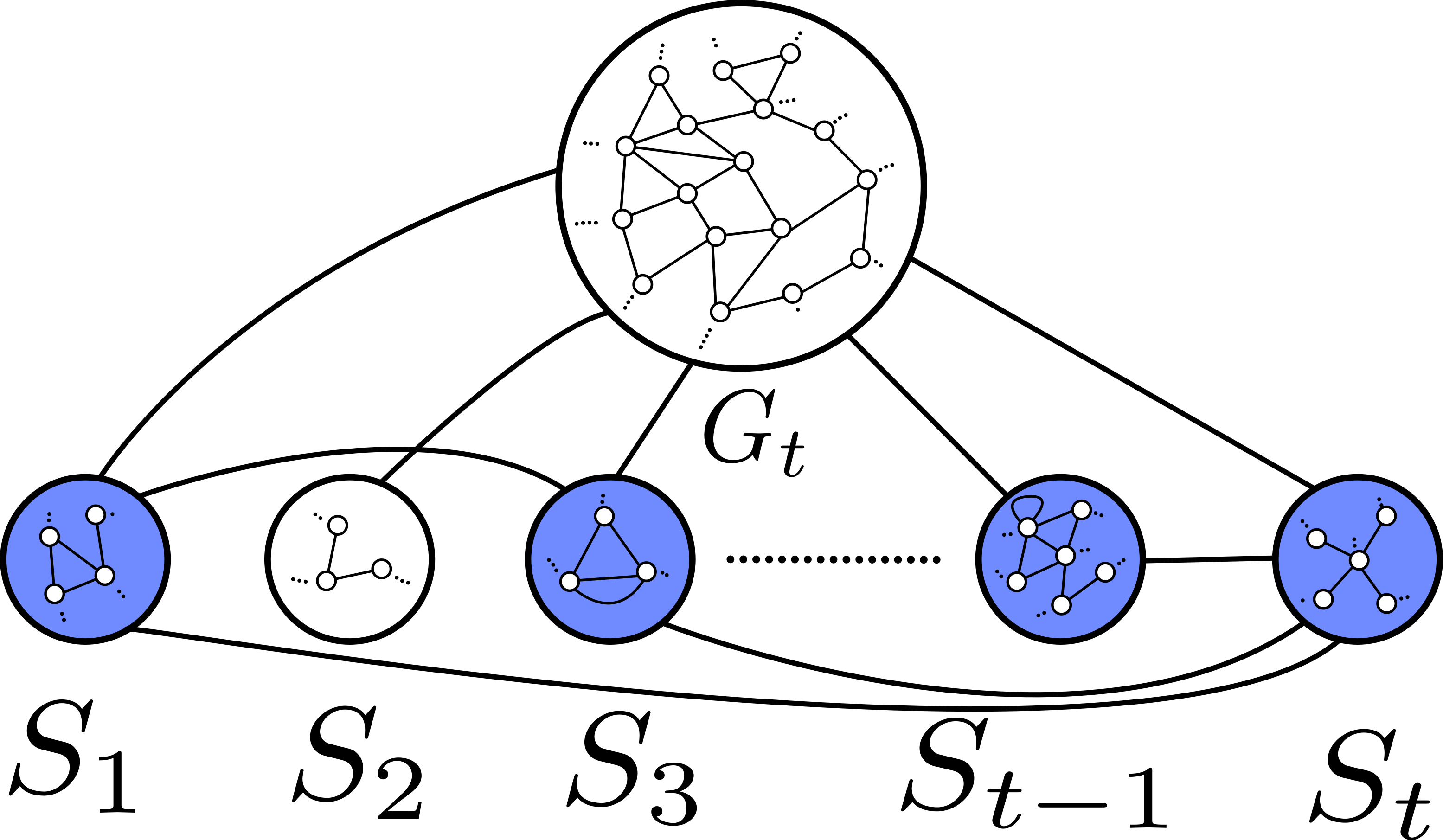}
    \caption{On the top: the subgraph $G_t \subset G$. On the bottom: the connected components $S_1,\cdots,S_t$ that have been removed by the process. The subset $S_{\mathcal{J}}$ is the blue part.}
    \label{proof_section4}
\end{figure}

By definition of the process, for every $1 \le i\leq t$, we have
    \begin{equation}\label{ineq_Si}
       \keps \vol_{ {G}}(S_i)> \keps \vol_{ {G_{i-1}}}(S_i)> \sum_{j>i} A_{i,j}+B_i\geq B_i
    \end{equation}

Now we can sum these inequalities for $i\in\J$ to obtain
\begin{equation}
   \keps\vol_{ {G}}(S_\J)>B_\J=e_{ {G}}(S_\J,G_t)=e_{ {G}}(S_\J,G\setminus S_\J) 
\end{equation}
where the last equality holds by definition of $\J$.

It remains to show that $\vol_{ {G}}(S_\J)<\frac{\vol(G)}{2}$.
First, by induction, we know that $\bigcup_{i<t}S_i\subset\mathbf{Isol}_{\keps}(G)\subset\mathbf{Isol}_{\kappa}(G)$, hence
\[\sum_{i<t}\vol_{ {G}}(S_i)<\eps n.\]
This entails by~\eqref{ineq_Si} that
\begin{equation}\label{uuu}
    \sum_{i<t}B_i<\keps\eps n.
\end{equation}
What's more, since $\vol_{ {G}}(S_t)<\vol_{ {G}}(G_t)$ and $2n=\vol_{ {G}}(S_t)+\vol_{ {G}}(G_t)+\sum_{i<t}\vol_{ {G}}(S_i)$, we get
\begin{equation}\label{vvv}
    \vol_{ {G}}(G_t)>(1-\eps)n>\eps n.
\end{equation}
On the other hand, using \cref{uuu,vvv} we get
\begin{align*}
    e_{ {G}}(G_t,G\setminus G_t)&=\sum_{i\leq t}B_i\\
    &\leq \keps\eps n+\keps \vol_{ {G}}(S_t)\\
    &\leq \keps \frac{\eps }{1-\eps}\vol_{ {G}}(G_t)+\keps\vol_{ {G}}(G_t)\\
    &=\kappa \vol_{ {G}}(G_t).
\end{align*}
and hence $\vol_{ {G}}(G_t)>\frac{\vol(G)}{2}$ otherwise it would be a \badset{}, and in particular we would have $\vol_{G}(\mathbf{Isol}_{\kappa}(G))>\eps n$  which contradicts the assumption of the theorem. Therefore $\vol_{ {G}}(S_\J)\leq \frac{\vol(G)}{2}$, hence $S_\J$ is a \badsetbis{} and $S_t\subset \mathbf{Isol}_{\keps}(G)$. This concludes the proof.

\end{proof}

\section{Expander subgraphs and duality: proof of~\cref{thm_dual}}

In this section, we show that, under mild degree conditions, having a large expander subgraph is a property that is stable under map duality: we prove~\cref{thm_dual}.

Recall that we start with a map $M$ on $n$ edges with face degrees bounded by $D>0$, such that its dual map $M^*$ contains a subgraph $G^*$ on at least $(1-\eps)n$ edges that is a \kexp{}.

Let $G$ be the subgraph of $M$ generated by edges that are dual to edges of $G^*$.  It is clear that $G$ contains at least $(1-\eps)n$ edges, it remains to show that it is an expander.

\begin{rem}
	With this construction, $G$ is not necessarily an induced subgraph of $M$. However, it was shown in \cite{LS20} that having large expander subgraphs, induced or not, is equivalent (up to changing $\kappa$).
\end{rem}

\begin{figure}[h]
    \centering
    \includegraphics[width=0.8\linewidth]{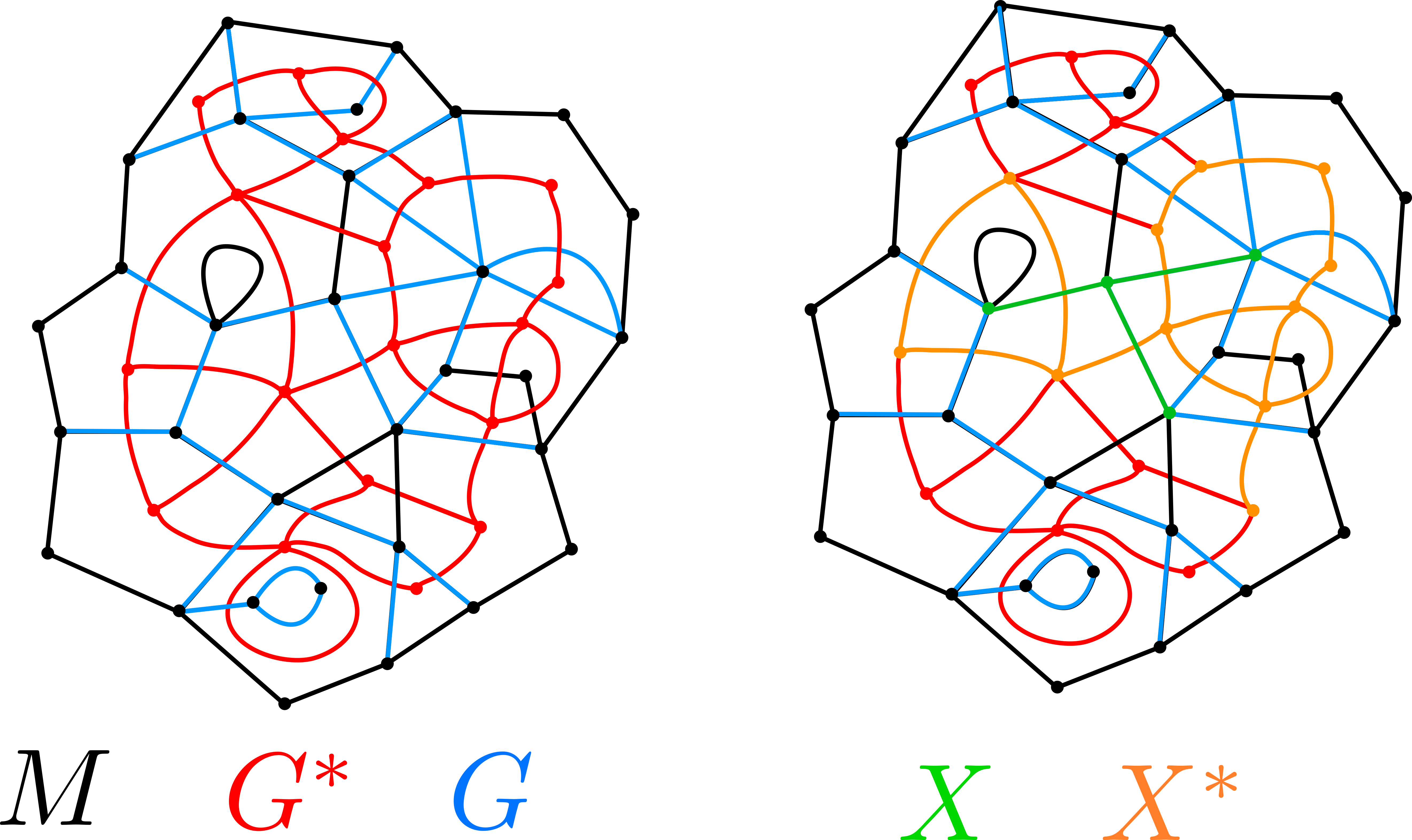}
    \caption{On the left: the map $M$ is represented in black. The subgraph $G^{*}$ of $M^{*}$ is represented in red. The subgraph $G$ of $M$ is represented in blue. On the right: the subgraph $X$ of $G$ is represented in green. The subgraph $X^{*}$ of $G^{*}$ is represented in orange.}
    \label{MGandGdual}
\end{figure}

\begin{prop}\label{prop_dual_expander}
   The graph $G$ is a $\frac{\kappa}{8D}$-expander.
\end{prop}
Let us fix $X\subset V(G)$ with $\vol_G(X)\leq \vol_G(G\setminus X)$.
In order to prove Proposition~\ref{prop_dual_expander}, we want to show that
\begin{equation}\label{eq_dual_expansion}
    e_{ {G}}(X,G\setminus X)\geq \frac{\kappa}{8D}\vol_G(X).
\end{equation}
If 
\[ e_{ {G}}(X,G\setminus X)\geq \tk \vol_G(X)\]
then \cref{eq_dual_expansion} is satisfied since $\tk\geq \frac{\kappa}{8D}$ (by definition of expansion, we necessarily have $\kappa\leq 1$).
Assume now that
\begin{equation}\label{eq_kappa_tilde}
    e_{ {G}}(X,G\setminus X)\leq \tk \vol_G(X).
\end{equation}

Let $X^*$ be the set of faces of $G^*$ that are incident to at least one vertex of $X$ (see Figure~\ref{MGandGdual}). We will compare the expansion of $X$ to the expansion of $X^*$.

We start by comparing volumes.
\begin{lem}\label{lem_ineq_volumes}
We have the inequality
    \[\vol_G(X)\leq \vol_{G^*}(X^*).\]
\end{lem}
\begin{proof}
Pick a convention for duality of oriented edges.
Let $\vec{e}$ be an oriented edge of $G$ whose starting point is in $X$ (notice that such oriented edges are counted by $\vol_G(X)$). Its dual $\vec{e}^*$ belongs to $G^*$ and has its starting point in $X^*$ by definition of $G^*$ and $X^*$. This defines an injective operation and the lemma follows.
\end{proof}

Then we compare the number of outgoing edges.
\begin{lem}\label{lem_ineq_outgoing}
We have the inequality
\begin{equation}
   e_{ {G}}(X,G\setminus X)\geq \frac 1 {2D}e_{ {G^{*}}}(X^*,G^*\setminus X^*).
\end{equation}    
\end{lem}

\begin{proof}
    
Let $F\subset X^*$ be the faces that are adjacent to at least one face not in $X^*$ (see Figure~\ref{F}).

First, by definition of $F$, we have

\begin{equation}
    e_{ {G^{*}}}(X^*,G^*\setminus X^*)=e_{ {G^{*}}}(F,G^*\setminus X^*)\leq \vol_{G^*}(F).
\end{equation}

\begin{figure}[h]
    \centering
    \includegraphics[width=0.4\linewidth]{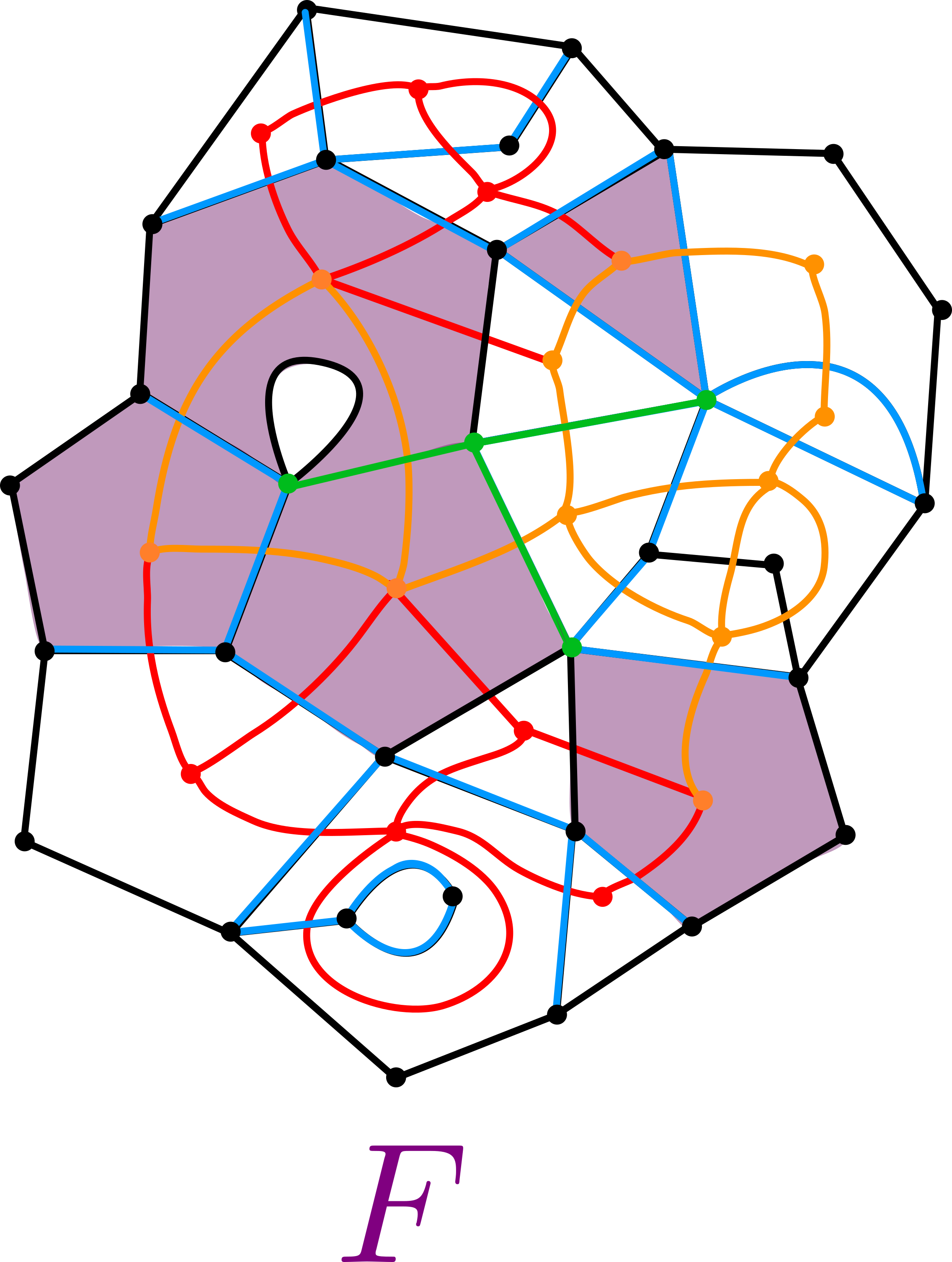}
    \caption{We represent in purple the set $F \subset X^{*}$ in purple. Note that it exactly corresponds to the set of vertices in $X^{*}$ that are incident to a red edge.}
    \label{F}
\end{figure}

On the other hand, by definition, every face $f\in F$ is incident to a vertex of $X$ and a vertex of $G\setminus X$. What's more, all its incident edges in $M$ are also in $G$ (this is due to the fact that $G^*$ is an induced subgraph of $M^*$). Therefore, $f$ is incident to at least one edge of  $E_{ {G}}(X,G\setminus X)$, and since every such edge touches two faces of $F$ at most, we get
\begin{equation}
    e_{ {G}}(X,G\setminus X)\geq \frac{\#F}{2}\geq \frac{\vol_{G^*}(F)}{2D}
\end{equation}
where in the second inequality we used the fact that all faces have degree bounded by $D$.

Concatenating the previous two inequalities, we obtain the result.

\end{proof}

Finally, we need to make sure that the volume of $X^*$ is not too large, in order to be able to use expansion inequalities later.

\begin{lem}\label{lem_vol_not_too_big}
    We have the inequality
    \[\vol_{G^*}(X^*)\leq \frac{3}{4}\vol(G^*).\]
\end{lem}
\begin{proof}
    We write $X^*=F_1\sqcup F_2$, where $F_1$ are faces that are incident to edges of $E(X)$ only.
    It is direct that
    \[\vol_{G^*}(F_1)\leq 2e(X)\leq \vol_G(X).\]
    
    On the other hand, similarly as in the proof of~\cref{lem_ineq_outgoing}, every face of $F_2$ has to be incident to at least one edge between $X$ and $G\setminus X$, and in return each such edge is incident to at most two of these faces, therefore
    \[\#F_2\leq 2 e_{ {G}}(X,G\setminus X). \]
    Now, since face degrees are bounded by $D$:
    \[\vol_{G^*}(F_2)\leq 2D e_{ {G}}(X,G\setminus X)\leq \frac{\vol_{G}(X)}{2}\]
    where in the second inequality we used~\cref{eq_kappa_tilde}.
    Concatenating both inequalities we get
    \[\vol_{G^*}(X^*)\leq \frac{3}{2}\vol_G(X),\]
    and recall that $\vol_G(X)\leq \vol(G)/2$ and that $\vol(G)=\vol(G^*)$ by definition of $G$.
    
\end{proof}

We can now prove~\cref{eq_dual_expansion}
\begin{proof}
Since $G^*$ is a \kexp{}, we get the inequality
\[\frac{e_{ {G^{*}}}(X^*,G^*\setminus X^*)}{\min(\vol_{G^*}(X^*),\vol_{G^*}(G^*\setminus X^*))}\geq \kappa.\]
But by Lemma~\ref{lem_vol_not_too_big}, we have the inequality
\[\vol_{G^*}(G^*\setminus X^*)\geq \frac{\vol(G^*)}{4}\geq \frac{\vol_{G^*}(X^*)}{4}.\]
It follows that 
\[\min(\vol_{G^*}(X^*),\vol_{G^*}(G^*\setminus X^*))\geq  \frac{\vol_{G^*}(X^*)}{4}.\]
Therefore
\[\frac{e_{ {G^{*}}}(X^*,G^*\setminus X^*)}{\vol_{G^*}(X^*)}\geq\frac{\kappa}{4} .\]
On the other hand, by Lemmas~\ref{lem_ineq_volumes} and \ref{lem_ineq_outgoing}, we have
\[\frac{e_{ {G}}(X,G\setminus X)}{\vol_{G}(X)}\geq\frac 1 {2D}\frac{e_{ {G^{*}}}(X^*,G^*\setminus X^*)}{\vol_{G^*}(X^*)}.\]
This finishes our proof.

\end{proof}

\medskip

\bibliographystyle{abbrv}
\bibliography{bibli}

\end{document}